\documentclass[12pt]{article}

\usepackage{amsthm}
\usepackage{amsmath}
\usepackage{amssymb}
\usepackage{latexsym}
\title{Intrinsic Approximation on Cantor-like Sets, a Problem of Mahler}
\author{Ryan Broderick, Lior Fishman and Asaf Reich}
\date{October 2011}

\newcommand{\R}{{\mathbb{R}}}

\newcommand{\N}{{\mathbb{N}}}

\newtheorem{thm}{Theorem}[section]

\newtheorem{cor}[thm]{Corollary}
\newtheorem{conj}[thm]{Conjecture}
\newtheorem*{conj*}{Conjecture}
\newtheorem{lemma}[thm]{Lemma}

\newtheorem{defn}[thm]{Definition}

\begin{document}
\maketitle

\begin{abstract}
In 1984, Kurt Mahler posed the following fundamental question: How well can
irrationals in the Cantor set be approximated by rationals in the Cantor set?
Towards development of such a theory, we prove a Dirichlet-type theorem for this
intrinsic Diophantine approximation on Cantor-like sets, and discuss related
possible theorems and conjectures. The resulting approximation function is analogous
to that for $\mathbb{R}^d$, but with $d$ being the Hausdorff dimension of the
set, and logarithmic dependence on the denominator instead.

\end{abstract}

\section{Introduction}

The Diophantine approximation theory of the real line is classical, extensive,
and essentially complete as far as characterizing how well real numbers can be
approximated by rationals (\cite{S} is a standard reference). The basic
result on approximability of all reals is
\begin{thm}[Dirichlet's Approximation Theorem] 

For each $x \in \mathbb{R}$, and for any $Q \in \mathbb{N}$, 
there exist $\frac{p}{q}\in \mathbb{Q}$ with $1\leq q \leq Q$, such that
$$\bigl|x-\frac{p}{q}\bigr|<\frac{1}{qQ}.$$
\end{thm}
\begin{cor} For every $x \in \mathbb{R}$, 
\begin{equation}
\label{dirich}
\bigl|x-\frac{p}{q}\bigr|<\frac{1}{q^2}\text{ for infinitely many }\frac{p}{q}\in\mathbb{Q}.
\end{equation}
\end{cor}

That this rate cannot be improved significantly for all reals is evidenced by
the existence of badly-approximable numbers (see definition \ref{BA}).
Furthermore, the rate cannot be improved for most
irrationals, in the sense that the set of very well approximable numbers (see definition \ref{VWA})
is null. In addition, the subject of approximating points on fractals by rationals has been extensively studied in recent
years; see for example \cite{F}, \cite{KW}, and  \cite{KTV} for badly approximable numbers
and \cite{KLW} and \cite{W} for very well approximable numbers.

In 1984, K. Mahler published a paper entitled ``Some suggestions for further research'' \cite{M},
in which he writes: ``At the age of 80 I cannot expect to do much more mathematics. I may however, state a number of questions where perhaps further research might lead to interesting results.'' One of these questions was regarding intrinsic and extrinsic approximation on the Cantor set. In Mahler's words,
``how close can irrational elements of Cantor's set be approximated by rational numbers

\begin{enumerate}
\item	In Cantor's set, and
\item	By rational numbers not in Cantor's set?''

\end{enumerate}

As far as the authors
are aware, there has been very little research on such intrinsic
Diophantine approximation on fractals, though there is such for
algebraic varieties (see for example \cite{GGN}). It is also worth noting that while 
addressing non-intrinsic questions,
the authors of \cite{LSV} used intrinsic techniques.
Studying the intersection of $C$ with numbers
approximable by rationals $\frac{p}{q}$ with $q$ of the specific form $3^n$, they
obtain lower bounds for general, non-intrinsic approximation. They point out
that this type of rational approximation to a point of $C$ will necessarily lie
in $C$. Thus some of their results are \emph{de
facto} also results giving lower bounds on approximations on $C$  by any rationals in $C$.
 However, these results are not likely to be optimal since only Cantor rationals
 with denominators of the form $3^n$ are considered, so
 we wish to consider intrinsic approximation as a topic in and of itself. We present some
initial steps towards a theory determining the analogous optimal rate of
intrinsic approximation for the Cantor set (and those constructed similarly).\\

The paper is organized as follows. In the following section we present our
main result. Section 3 is devoted to formulating our conjecture regarding the
distribution of rationals on the Cantor set, which if true would allow us to conclude that our
main theorem is optimal. The last section is devoted to extrinsic approximation
on the Cantor set, i.e., approximating irrationals on the Cantor by rationals not on the Cantor set . 

\section{Dirichlet type theorem for Cantor sets}

\noindent Let $C$ denote the Cantor-like set consisting of numbers in $I=[0,1]$
which can be written in base $b>2$ using only the digits in $S\subset
\{0,1,\dots,b-1\}$, where $\bigl|S\bigr|=a>1$.
Obviously this is equivalent to partitioning $I$ into $b$ equal subintervals,
only keeping those indexed by S, and successively continuing this on each
remaining subinterval. (The usual middle-thirds set is given by $b=3$,
$S=\{0,2\}$.) We will assume that $S$ is a proper subset of $\{0,1,\dots,b-1\}$,
i.e. that $C\neq I$, as the case $C=I$ is simply the classical one.

Throughout the paper we denote by $\{x\}$ the fractional part of $x$ and by
$\lfloor x \rfloor$ the integer part. To state the theorem we also associate to
the set the number $b_0$, where $b_0$ is the least integer greater than $1$ such that $C$ is
invariant under $x \mapsto \{b_0 x\}$. As we shall see, this is simply equal to
$b$ except when the given definition of $C$ is redundant in a certain sense,
in which case we still have $b = b_0^r$ for some $r\in\N$.
\begin{thm}[Dirichlet for Cantor sets] \label{main}
Let $d = \dim C$.
For every $x\in C$ and every $Q$ of the form $b^n$, there exists $p/q \in
\mathbb{Q} \cap C$ with $q \leq b_0^{Q^d}$ such that
$$\bigl|x-\frac{p}{q}\bigr|<\frac{1}{qQ}.$$
\end{thm}

It follows from the statement that for \emph{any} $Q \in
\mathbb{N}$, the same statement holds with the bound multiplied by b (by letting
$b^n \leq Q < b^{n+1}$). As an immediate corollary to Theorem \ref{main} we get:

\begin{cor}
\label{maincor}
For all $x\in C,$ there exist infinitely many solutions $p\in\mathbb{N}$,
$q\in\mathbb{N}$, $\frac{p}{q}\in C$ to \[
\bigl|x-\frac{p}{q}\bigr|<\frac{1}{q(\log _{b_0} q)^{1/d}}.\]
\end{cor}

Also notice that the approximation function's asymptotic behavior gets
better as $d \to 0$, even though the first such $q$ whose existence we prove can
tend to infinity as $b$ does.

Before we begin the proof, we first need a characterization of the rationals in
$C$:
\begin{lemma}

A rational number is in $C$ if and only if it can be written either as a
terminating base-$b$ expansion 
(left end points in the construction) or as
\begin{equation}
\label{Cantor1}
\frac{\bigl(\sum_{i=0}^{k+l-1}c_{i}b^{k+l-1-i}-\sum_{j=0}^{k-1}c_{j}b^{k-1-j}
\bigr)}{b^{k+l}-b^{k}}.
\end{equation}

\noindent where $l,k\in\mathbb{N}$, and  $c_{i}\in S $.

\noindent Equivalently (\ref{Cantor1}) can also be expressed  in terms of
base-$b$ expansions, i.e.,

\begin{equation}
\label{Cantor2}
\frac{
\bigl((c_{0}c_{1}...c_{k+l-1})_{b}-(c_{0}c_{1}...c_{k-1})_{b}\bigr)}{b^{k+l}-b^{
k}}.
\end{equation}

\end{lemma}

\begin{proof}
A rational in $C$ has either a terminating $b$-ary expansion (consisting of
digits from $S$) or an eventually periodic one. If it is purely periodic of
period $l$, it has the form \[
\sum_{n=1}^{n=\infty}{\frac{b^{l-1}x_0+b^{l-2}x_1+...+x_{l-1}}{(b^l)^n}} =
\frac{b^{l-1}x_0+b^{l-2}x_1+...+x_{l-1}}{b^l-1} , \] where the digits $x_i$ are
in $S$. If the rational is not purely periodic then one must insert some number
$k$ of initial zeros and then add the initial terminating expansion of length
$k$, so we obtain the form \[
 \frac{b^{l-1}x_0+b^{l-2}x_1+...+x_{l-1}}{(b^l-1)b^k} +
\frac{(b^l-1)(b^{k-1}y_0+b^{k-2}y_1+...+y_{k-1})}{(b^l-1)b^k} \] for some digits
$y_i \in S$. Rearranging this gives the result.
\end{proof}

\begin{proof} [Proof of theorem \ref{main}]
Let $x \in C \setminus \mathbb{Q}$. Given any $n \in \mathbb{N}$, let
$Q=b^{n}$. Denote by $M=(m_i)$ the semigroup of positive integer multiplication
maps mod 1 leaving $C$ invariant. We order the elements in increasing order from
$1=m_0$.

\noindent
There are $a^{n}$ possibilities for the first $n$ digits in the
$b$-ary expansion of ${\{qx\}}$. Consider the elements of $C$ given by $0$,
$\{m_0 x\}$,...$\{m_{a^n-1}x\}$. By the pigeonhole principle either there exist
$0 \leq q,q' <a^n$ such that $\{m_q x\}$ and $\{m_{q'}x\}$ have the same first
$n$ digits, or the same holds for some $0 \leq q <a^n$ and $0$. That is, they
are in the same interval of the ${n}$-th stage of $C$'s
construction. Assuming the former, it follows that $|\{m_{q'}x\}-\{m_q
x\}|<\frac{1}{Q}$. Rewriting gives
\[
\bigl|m_{q'} x-\left\lfloor m_{q'} x\right\rfloor -m_q x+\left\lfloor m_q
x\right\rfloor \bigr|<\frac{1}{Q}.\]

Setting $p=\left\lfloor m_{q'}x\right\rfloor -\left\lfloor m_q x\right\rfloor $
we get
\[
\bigl|x-\frac{p}{m_q'-m_q}\bigr|<\frac{1}{(m_q'-m_q)Q}.\]

If one of the above values is 0 rather than an $m_i$ the calculation is trivial.
Suppose $M$ were generated by more than one element. Then by H.\ Furstenberg's seminal result in \cite{Fu},
stating
that the only infinite closed subset of $\mathbb{R}/\mathbb{Z}$ invariant under
$M$ is $\mathbb{R}/\mathbb{Z}$ itself, we would have to have $C=I$. But we assumed $C\neq I$, so that cannot be true.
(Notice however that we recover Dirichlet's original theorem in this case, since
then $M = \N$, so $m_{q'}-m_q < a^n=Q$.)
So $M$ has a single generator, which we will denote $b_0$, and then $b=b_0^r$ for some $r$,
and $a=a_0^r$. $C$ is also the set constructed in the corresponding
way for $b_0$ and some $S_0$ (namely, the unique subset of $\{0,1,\dots,b_0-1\}$
such that $S$ is the set $S_0^r$ of $r$-fold concatenations of elements of $S_0$). 
Then $m_{a^n-1} = b_0^{a^n-1}$.
Since $m_{q'}-m_q=b_0^{k}(b_0^{d}-1)$ for some integers $k,d$, following
(\ref{Cantor2}) it suffices to observe that $p=\left\lfloor
b_0^{k+d}x\right\rfloor -\left\lfloor b_0^k x\right\rfloor$. Thus 
$\frac{p}{m_{q'}-m_q} \in C$, and $m_{q'}-m_q < b_0^{a^{n}}=b_0^{Q^d}$.
\end{proof}

It is worth noticing a similarity in form between Corollary \ref{maincor} and the more general form of Dirichlet's theorem, stating that points in $\R^d$ are approximable according to the function $\frac{1}{qq^{1/d}}$. However there is a ``gap'' between the approximation functions in these two theorems -- the $q^{1/d}$ from the classic case is replaced by a $(\log_{b_0} q)^{1/d}$ in the fractal case, so that the intrinsic approximation function on $C$ does not tend to $q^{-2}$ even as $\dim C$ tends to $1$.
The reason for this dichotomy is evident from the proof, which hinges on the growth of $m_k$ in terms of $k$:
When $M$ has two generators, $m_k = k$, corresponding to the classical case, and when $M$ has a single generator, $m_k = b_0^k$,
corresponding to the fractal case. These two possibilities for the asymptotics of $m_k$ -- either exponential or linear -- lead to the corresponding differences in the asymptotics of the approximation functions.

\section{Further investigation, conjectures and consequences}

The obvious next step would be to check whether this is the optimal rate of
approximation for $C$. We first recall the following definitions:
\begin{defn}
\label{BA}
The set of badly approximable numbers ({\bf{BA}}) consists of all reals $x$ such that for some $c(x)>0$ \[
\bigl| \frac{p}{q} - x \bigr| > \frac{c(x)}{q^2} \hspace{2mm} \text{ for all } p/q \in
\mathbb{Q}. \] 
\end{defn}
We thus define, relative to the approximation function proven above,
the {\sl{intrinsic}} set of badly approximable numbers, denoted by ${\bf{BA}}_{C}$, as those in $C$ satisfying \[
\bigl| \frac{p}{q} - x \bigr| > \frac{c(x)}{q(\log _{b_0} q)^{1/d}} \hspace{2mm}
\text{ for all } p/q \in \mathbb{Q}\cap C. \]

\begin{defn}
\label{VWA}
The set  of very well approximable numbers ({\bf{VWA}}) is
the set of all reals $x$ satisfying for some positive $\epsilon(x)$
$$
\bigl| \frac{p}{q} - x \bigr| < \frac{1}{q^{2+\epsilon(x)}}
$$
for infinitely
many rationals $p/q$.
\end{defn}
We thus define the \emph{intrinsic} {\bf{VWA}} numbers, denoted ${\bf{VWA}}_c$, as those in
$C$ satisfying
$$
\bigl| \frac{p}{q} - x \bigr| < \frac{1}{q^{1+\epsilon(x)}(\log _{b_0} q)^{1/d}}\text{ for infinitely many }
	p/q\in C,
$$ 
or equivalently, since $\epsilon(x) > 0$ is arbitrary and can depend on $x$, as those
in $C$ satisfying
\begin{equation}
\label{intrinsic}
\bigl| \frac{p}{q} - x \bigr| < \frac{1}{q^{1+\epsilon(x)}}\text{ for infinitely many }
	p/q\in \mathbb{Q}\cap C.
\end{equation}
One could definitively say that our approximation function is optimal if one proved that intrinsic
${\bf{BA}}_C$ is nonempty,
and one could conclude that the power of $q$ in the denominator is correct
if ${\bf{VWA}}_C$ has zero
$d$-dimensional Hausdorff measure.

These results appear difficult to achieve, however, without
having a deeper understanding of the rationals in $C$. For the real line, there
exists a large arsenal of useful information regarding distribution properties
of rationals, their quantity within bounds on $q$, etc., whereas for $C$ nothing
is even known about which denominators can appear (in reduced form, of course;
the expression (\ref{Cantor2}) is not useful here since it is not reduced). In
fact Mahler points out basically the same fundamental difficulty. To obtain the
desired results, we believe a major new piece of information will be necessary, such as
a sharp bound either on how many rationals appear in $C$ with (reduced) 
denominators in a given range
or on how these rationals ``repel" each other as a function of $q$.
We nonetheless make the following conjecture for the standard Cantor
set supported by computer calculations (see table 1 below).

\begin{conj}
\label{cantorconj}
Let $N(s,t) = \|\{\frac{p}{q} \in C, gcd(p,q)=1, s \leq q \leq t\}\|$. Then
$N(3^n,3^{n+1}) \in O(2^{(1+\epsilon)n})$ for all $\epsilon>0.$
\end{conj}

Notice that defining 
$\phi (n)=\dfrac{1}{2}
\dfrac{N(3^n,3^{n+1})}{N(3^{n-1},3^n)}$,
the conjecture is equivalent to the convergence of $\phi (n)$ to 1. The table below
sums up our computer based results:\\

\begin{table}[h]
\caption{} 
\centering
\begin{tabular}{|c|| c| c| c| c| c| c| c| c| c| c| c|}
\hline
n&4&5&6&7&8&9&10&11 \\ 
\hline
$\phi (n)$& 1.808& 1.064& 1.32& 1.18& 1.258 & 1.057& 1.176& 1.063\\

\hline

\end{tabular}
\label{tab:hresult}
\end{table}

The conjecture is significant, implying that the power of $q$ in our theorem cannot be improved.

\begin{cor} [\bf{of conjecture \ref{cantorconj}}]
$\mu(\text{{\bf{VWA}}}_{C})=0$, where $\mu$ is the $d$-dimensional Hausdorff measure
restricted to $C$.
\end{cor}
\begin{proof}
Let $V_{\epsilon}$ denote the subset of $\text{{\bf{VWA}}}_{C}$ satisfying (\ref{intrinsic})
for a particular $\epsilon(x)=\epsilon$, and observe that $V_{\epsilon}$ is the
lim sup of the sets \begin{equation} S_n=
\bigcup_{p/q \in C, 3^n \leq q < 3^{n+1}}
\left(\frac{p}{q}-\frac{1}{q^{1+\epsilon}},
\frac{p}{q}+\frac{1}{q^{1+\epsilon}}\right) .\end{equation}
By the Borel-Cantelli lemma, $\mu(V_{\epsilon})=0$ if $\sum_{n=0}^{\infty}
\mu(S_n)<\infty$.
The number of intervals in the union $S_n$ is by assumption in
$O(2^{(1+\epsilon')n})$ for all $\epsilon'>0$. The radius of each interval is
$O(\frac{1}{3^{(1+\epsilon)n}})$, and $\mu$ satisfies a $d$-power law, so we
have that $\mu(B(x,r))$ is $O(r^d)$ for $x \in C$ and small $r$. So the measure
of each interval is $O(\frac{1}{2^{(1+\epsilon)n}})$, and \begin{equation}
\mu(S_n) \in O\left(\frac{2^{(1+\epsilon')n}}{2^{(1+\epsilon)n}}\right) 
= O(2^{(\epsilon'-\epsilon)n}) \hspace{2mm} \text{ for all } \epsilon'>0. \end{equation}  
Thus we take $\epsilon' <\epsilon$, and obtain that $\sum_{n=0}^{\infty}
\mu(S_n)<\infty$ as desired. Finally observe $\text{{\bf{VWA}}}_{C}=\cup_{m=1}^{\infty}
V_{1/m}$, so $\mu(\text{{\bf{VWA}}}_{C})=0$.\\
\end{proof}

At this point the reader might be thinking that a stronger bound on $N$ would
allow one to prove that theorem \ref{main} is even sharper: specifically,
if $N(3^n,3^{n+1}) \in O(2^n)$, then a proof of similar type
would show that the superset
$$\big\{x : \bigl| \frac{p}{q} - x \bigr| < \frac{1}{q(\log _{b_0} q)^{1/d+\epsilon(x)}}\text{ for infinitely many } p/q\in C\big\} \supset {\bf{VWA}}_C$$
is $\mu$-null.
However this avenue is not possible, as a forthcoming followup paper by Noam
Solomon and Barak Weiss discussing Conjecture \ref{cantorconj} and partial
progress towards it will prove that this stronger bound on $N$ does not hold.

\section{Extrinsic approximation}
As mentioned in the introduction, Mahler also raised the complementary \emph{extrinsic} question: how well
can a rational in $C$ be approximated by rationals \emph{outside} of $C$? To
this end, consider a weaker variant of Conjecture \ref{cantorconj}:
\begin{conj}
 \label{weak}
$N(3^n,3^{n+1}) \in O(2^{\gamma n})$ for some $\gamma<2$.
\end{conj}
Note that a confirmation of this conjecture (for which our calculations are
quite strong evidence) would be sufficient to show, by an argument similar to the one above, that the set
\begin{equation}
\label{fact}
\big\{x\in C : \bigl| \frac{p}{q} - x \bigr| < \frac{1}{q^2} \text{ for infinitely many } \frac{p}{q} \in C\big\}
\end{equation}
 is $\mu$-null. In turn, consider the set of $x \in C$ such
that \begin{equation} \label{wa}
\text{For all } \epsilon>0\text{, there exist infinitely many }
\frac{p}{q} \notin C \text{ satisfying } \bigl| \frac{p}{q} - x \bigr| <
\frac{\epsilon}{q^2}. 
\end{equation} 
We denote this extrinsic analogue of {\bf{WA}}, the complement of {\bf{BA}}, by ${\bf{WA}}_{C^c}$
and its complement by ${\bf{BA}}_{C^c}$. Thus (\ref{fact}), combined
with Corollary 1.10 of \cite{EFS} stating that ${\bf{BA}}\cap C$ is
$\mu$-null, show the following:

\begin{cor} [\bf{of conjecture \ref{weak}}] ${\bf{WA}}_{C^c}$ is of $\mu$-full measure.
\end{cor} 
We note that the sharpness in either direction of this statement is already
known by \cite{KW} and \cite{W}.



\bibliographystyle{alpha}

\noindent \texttt {Ryan Broderick - 
Department of Mathematics,
Northwestern University, 2033 Sheridan Road,
Evanston IL 60208-2370, USA}

\texttt{ryan.broderick@northwestern.edu} \\

\noindent \texttt {Lior Fishman - Department of Mathematics, University of North Texas,
1155 Union Circle 311430,
Denton, TX 76203-5017, USA}

\texttt{lfishman@unt.edu} \\

\noindent \texttt {Asaf Reich - Department of Mathematics, Brandeis University, 415 South Street, Waltham MA 02454-9110, USA} 

\texttt{azmreich@brandeis.edu}

\end{document}